\documentclass{amsart}%%%%%%%%%%Submitted to Acta Arithmetica, August, 2015; modified by Shabnam, Dec 22, 2015,
\usepackage{amssymb, latexsym}
\usepackage{url}

\DeclareMathOperator{\CM}{\mathrm{CM}}

\DeclareMathOperator{\norm}{\mathrm{norm}}
\DeclareMathOperator{\Norm}{\mathrm{Norm}}

\DeclareMathOperator{\Reg}{\mathrm{Reg}}

\DeclareMathOperator{\Tor}{\mathrm{Tor}}
\DeclareMathOperator{\Vol}{\mathrm{Vol}}

\begin{document}
 \bibliographystyle{plain}

 \newtheorem{theorem}{Theorem}[section]
 \newtheorem{lemma}{Lemma}[section]
 \newtheorem{corollary}{Corollary}[section]
 \newtheorem{conjecture}{Conjecture}
 \newtheorem{definition}{Definition}
 \newcommand{\mc}{\mathcal}
 \newcommand{\A}{\mc{A}}
 \newcommand{\B}{\mc{B}}
 \newcommand{\cc}{\mc{C}}
 \newcommand{\D}{\mc{D}}
 \newcommand{\E}{\mc{E}}
 \newcommand{\F}{\mc{F}}
 \newcommand{\G}{\mc{G}}
 \newcommand{\FN}{\F_n}
 \newcommand{\I}{\mc{I}}
 \newcommand{\J}{\mc{J}}
 \newcommand{\eL}{\mc{L}}
 \newcommand{\M}{\mc{M}}
 \newcommand{\eN}{\mc{N}}
 \newcommand{\qq}{\mc{Q}}
 \newcommand{\U}{\mc{U}}
 \newcommand{\V}{\mc{V}}
 \newcommand{\X}{\mc{X}}
 \newcommand{\Y}{\mc{Y}}
 \newcommand{\C}{\mathbb{C}}
 \newcommand{\R}{\mathbb{R}}
 \newcommand{\N}{\mathbb{N}}
 \newcommand{\Q}{\mathbb{Q}}
 \newcommand{\T}{\mathbb{T}}
 \newcommand{\Z}{\mathbb{Z}}
 \newcommand{\aA}{\mathfrak A}
 \newcommand{\bB}{\mathfrak B}
 \newcommand{\ff}{\mathfrak F}
 \newcommand{\ee}{\mathfrak E}
 \newcommand{\uU}{\mathfrak U}
 \newcommand{\fb}{f_{\beta}}
 \newcommand{\fg}{f_{\gamma}}
 \newcommand{\gb}{g_{\beta}}
 \newcommand{\vep}{\varepsilon}
 \newcommand{\vphi}{\varphi}
 \newcommand{\bo}{\boldsymbol 0}
 \newcommand{\bone}{\boldsymbol 1}
 \newcommand{\ba}{\boldsymbol a}
 \newcommand{\bb}{\boldsymbol b}
 \newcommand{\bc}{\boldsymbol c}
 \newcommand{\be}{\boldsymbol e}
 \newcommand{\bk}{\boldsymbol k}
 \newcommand{\bell}{\boldsymbol \ell}
 \newcommand{\bm}{\boldsymbol m}
 \newcommand{\bgamma}{\boldsymbol \gamma}
 \newcommand{\bt}{\boldsymbol t}
 \newcommand{\bu}{\boldsymbol u}
 \newcommand{\bv}{\boldsymbol v}
 \newcommand{\bx}{\boldsymbol x}
 \newcommand{\bwy}{\boldsymbol y}
 \newcommand{\bxi}{\boldsymbol \xi}
 \newcommand{\bbeta}{\boldsymbol \eta}
 \newcommand{\bw}{\boldsymbol w}
 \newcommand{\bz}{\boldsymbol z}
 \newcommand{\whG}{\widehat{G}}
 \newcommand{\oK}{\overline{K}}
 \newcommand{\oKt}{\overline{K}^{\times}}
 \newcommand{\oQ}{\overline{\Q}}
 \newcommand{\oq}{\oQ^{\times}}
 \newcommand{\oQt}{\oQ^{\times}/\Tor\bigl(\oQ^{\times}\bigr)}
 \newcommand{\ot}{\Tor\bigl(\oQ^{\times}\bigr)}
 \newcommand{\h}{\frac12}
 \newcommand{\hh}{\tfrac12}
 \newcommand{\dx}{\text{\rm d}x}
 \newcommand{\dbx}{\text{\rm d}\bx}
 \newcommand{\dy}{\text{\rm d}y}
 \newcommand{\dmu}{\text{\rm d}\mu}
 \newcommand{\dnu}{\text{\rm d}\nu}
 \newcommand{\dla}{\text{\rm d}\lambda}
 \newcommand{\dlav}{\text{\rm d}\lambda_v}
 \newcommand{\trho}{\widetilde{\rho}}
 \newcommand{\dtrho}{\text{\rm d}\widetilde{\rho}}
 \newcommand{\drho}{\text{\rm d}\rho}
 \def\today{\number\time, \ifcase\month\or
  January\or February\or March\or April\or May\or June\or
  July\or August\or September\or October\or November\or December\fi
  \space\number\day, \number\year}

\title[Heights]{Heights, Regulators and Schinzel's\\determinant inequality}
\author{Shabnam~Akhtari}
\author{Jeffrey~D.~Vaaler}
\thanks{This research was supported by NSA grant, H98230-12-1-0254.}

\address{Department of Mathematics, University of Oregon, Eugene, Oregon 97403 USA}
\email{akhtari@uoregon.edu}

\address{Department of Mathematics, University of Texas, Austin, Texas 78712 USA}
\email{vaaler@math.utexas.edu}
%\allowdisplaybreaks
\numberwithin{equation}{section}

\maketitle

%%%%%%%%%%%%%%%%%%%%%%%%%%%%%%%%%%%%%%%%%%%%%%%%%%%%%%%%%%%%%%
\begin{abstract}  We prove inequalities that compare the size of an $S$-regulator with a product of heights of
multiplicatively independent $S$-units.  Our upper bound for the $S$-regulator follows from a general upper bound for the determinant of
a real matrix proved by Schinzel.  The lower bound for the $S$-regulator follows from Minkowski's theorem on successive minima and
a volume formula proved by Meyer and Pajor.  We establish similar upper bounds for the relative regulator of an extension
$l/k$ of number fields.
\end{abstract}

%%%%%%%%%%%%%%%%%%%%%%%%%%%%%%%%%%%%%%%%%%%%%%%%%%%%%%%%%%%%%%%
\section{Introduction}

Let $k$ be an algebraic number field, $k^{\times}$ its multiplicative group of nonzero elements, and 
$h : k^{\times} \rightarrow [0, \infty)$ the absolute, logarithmic, Weil height.  If $\alpha$ belongs to $k^{\times}$ and $\zeta$ is a 
root of unity in $k^{\times}$, then the identity $h(\zeta \alpha) = h(\alpha)$ is well known.  It follows that the height $h$ is
constant on cosets of the quotient group
\begin{equation*}\label{intro2}
\G_k = k^{\times}/\Tor\bigl(k^{\times}\bigr).
\end{equation*}
Therefore the height is well defined as a map $h : \G_k \rightarrow [0, \infty)$. 

Let $S$ be a finite set of places of $k$ such that $S$ contains all the archimedean places.  Then
\begin{equation*}\label{intro7}
O_S = \big\{\gamma \in k : |\gamma|_v \le 1\ \text{for all places}\ v \notin S\big\}
\end{equation*}
is the ring of $S$-integers in $k$, and
\begin{equation}\label{intro9}
O_S^{\times} = \big\{\gamma \in k^{\times} : |\gamma|_v = 1\ \text{for all places}\ v \notin S \big\}
\end{equation}
is the multiplicative group of $S$-units in the ring $O_S$.  We write
\begin{equation}\label{intro10}
\Tor\bigl(O_S^{\times}\bigr) = \Tor\bigl(k^{\times}\bigr)
\end{equation}
for the torsion subgroup of $O_S^{\times}$, which is also the torsion subgroup of the multiplicative group $k^{\times}$.
As is well known, (\ref{intro10}) is a finite, cyclic group of even order, and
\begin{equation}\label{intro12}
\uU_S(k) = O_S^{\times}/\Tor\bigl(O_S^{\times}\bigr) \subseteq \G_k
\end{equation}
is a free abelian group of finite rank $r$, where $|S| = r + 1$.

In this paper we establish simple inequalities between the $S$-regulator $\Reg_S(k)$ and products of the form
\begin{equation*}\label{intro14}
\prod_{j = 1}^r \bigl([k : \Q] h(\alpha_j)\bigr),
\end{equation*}
where $\alpha_1, \alpha_2, \dots , \alpha_r$ are multiplicatively independent elements in the group $\uU_S(k)$.  

\begin{theorem}\label{thmintro1}  Let the multiplicative group of $S$-units $O_S^{\times}$ have positive rank $r$, and let
$\alpha_1, \alpha_2, \dots , \alpha_r$ be multiplicatively independent elements in the free group $\uU_S(k)$.  If 
$\aA \subseteq \uU_S(k)$ is the multiplicative subgroup generated by $\alpha_1, \alpha_2, \dots , \alpha_r$, then
\begin{equation}\label{intro58}
\Reg_S(k) \bigl[\uU_S(k) : \aA \bigr] \le \prod_{j = 1}^r \bigl([k : \Q] h(\alpha_j)\bigr).
\end{equation}
\end{theorem}

A special case of (\ref{intro58}) occurs when $S$ is the collection of all archimedean places of $k$.  We write 
$O_k$ for the ring of algebraic integers in $k$, and $O_k^{\times}$ for the multiplicative group of units in $O_k$.  If $k$ 
is not $\Q$, and $k$ is not an imaginary quadratic extension of $\Q$, then the quotient group
\begin{equation*}\label{intro60}
\uU(k) = O_k^{\times}/\Tor\bigl(O_k^{\times}\bigr) \subseteq \G_k
\end{equation*}
is a free abelian group of positive rank $r$, where $r + 1$ is the number of archimedean places of $k$.  It is known from
work of Remak \cite{remak1932}, \cite{remak1952}, and Zimmert \cite{zimmert1981}, that the regulator $\Reg(k)$ is bounded from below 
by an absolute constant.  Further, Friedman \cite{friedman1989} has shown that $\Reg(k)$ takes its minimum value at the unique 
number field $k_0$ having degree $6$ over $\Q$, and having discriminant equal to $-10051$.  Thus by Friedman's result we have
\begin{equation}\label{intro61}
0.2052 \dots = \Reg(k_0) \le \Reg(k)
\end{equation}
for all algebraic number fields $k$.  Combining the inequalities (\ref{intro58}) and (\ref{intro61}) leads to the following 
explicit lower bound.

\begin{corollary}\label{corintro1}  Assume that $k$ is not $\Q$, and $k$ is not an imaginary quadratic extension of $\Q$, so that $\uU(k)$ 
has positive rank $r$.  Let $\alpha_1, \alpha_2, \dots , \alpha_r$ be multiplicatively independent elements in $\uU(k)$.   If 
$\aA \subseteq \uU(k)$ is the subgroup generated by $\alpha_1, \alpha_2, \dots , \alpha_r$, then
\begin{equation}\label{intro62}
(0.2052 \cdots )\bigl[\uU(k) : \aA \bigr]   \le \prod_{j = 1}^r \bigl([k : \Q] h(\alpha_j)\bigr).
\end{equation}
\end{corollary}

Let $k$ be an algebraic number field such that the unit group $O_k^{\times}$ has positive rank $r$.  The inequality 
(\ref{intro62}) implies that each collection $\alpha_1, \alpha_2, \dots , \alpha_r$ of multiplicatively independent units must 
contain a unit, say $\alpha_1$, that satisfies
\begin{equation}\label{intro64}
(0.2052 \cdots ) \le [k : \Q] h(\alpha_1).
\end{equation}
A result of this sort was proposed by Bertrand \cite[comment (iii), p. 210]{bertrand1997}, who observed that it would follow
from an unproved hypothesis related to Lehmer's problem.

In a well known paper Lehmer \cite{lehmer1933} posed the problem, reformulated in the language 
and notation developed here, of deciding if there exists a positive constant $c_0$ such that the inequality
\begin{equation}\label{intro73}
c_0 \le [k : \Q] h(\gamma)
\end{equation}
holds for all elements $\gamma$ in $k^{\times}$, which are not in $\Tor\bigl(k^{\times}\bigr)$.  If $\gamma \not= 0$ is not a unit,
then it is easy to show that
\begin{equation*}\label{intro74}
\log 2 \le [k : \Q] h(\gamma).
\end{equation*}
Hence the proposed lower bound (\ref{intro73}) is of interest for non-torsion elements $\gamma$ in the unit group $O_k^{\times}$,
or equivalently, for a nontrivial coset representative $\gamma$ in $\uU(k)$.  The inequality (\ref{intro62})
provides a solution to a form of Lehmer's problem on average.  Further information about Lehmer's problem is given in 
\cite[section 1.6.15]{bombieri2006} and in \cite{smyth2008}.

In section 3 we give an analogous upper bound for the relative regulator associated to an extension $l/k$ of algebraic
number fields.

We will show that the inequality (\ref{intro58}) is sharp up to a constant that depends only on the rank $r$ of the group 
$\uU_S(k)$, but {\it not} on the underlying field $k$.  Related results have been
proved by Brindza \cite{brindza1991}, Bugeaud and Gy\H{o}ry \cite{bugeaud1996}, Hajdu \cite{hajdu1993},
and Matveev \cite{matveev1994}, \cite{matveev2005}.  More general inequalities that apply to arbitrary
finitely generated subgroups of $\oq$ were obtained in \cite[Theorem 1 and Theorem 2]{vaaler2014}.  The inequality 
(\ref{intro88}) that we prove here is sharper but less general, as it applies only to subgroups of a group of $S$-units
having maximum rank.

\begin{theorem}\label{thmintro2}  Let the multiplicative group of $S$-units $O_S^{\times}$ have positive rank $r$, and let
$\aA \subseteq \uU_S(k)$ be a subgroup of rank $r$.  Then there exist multiplicatively
independent elements $\beta_1, \beta_2, \dots , \beta_r$ in $\aA$, such that
\begin{equation}\label{intro88}
\prod_{j = 1}^r \bigl([k : \Q] h(\beta_j)\bigr) \le \frac{2^r (r!)^3}{(2r)!} \Reg_S(k) \bigl[\uU_S(k) : \aA \bigr]. 
\end{equation}
\end{theorem}

We note that if $r = 2$ then (\ref{intro58}) and (\ref{intro88}) imply that the multiplicatively independent elements $\beta_1$ and $\beta_2$
contained in the subgroup $\aA \subseteq \uU_S(k)$ satisfy the inequality
\begin{equation*}\label{intro91}
 \Reg_S(k) \bigl[\uU_S(k) : \aA \bigr] \le \bigl([k : \Q] h(\beta_1)\bigr)\bigl([k : \Q] h(\beta_2)\bigr) 
                                                                \le \tfrac43  \Reg_S(k) \bigl[\uU_S(k) : \aA \bigr].
\end{equation*}
It follows that $\beta_1$ and $\beta_2$ form a basis for the group $\aA$.  More generally, by using a well known lemma proved
by Mahler \cite{mahler1938} and Weyl \cite{weyl1942} (see also \cite[Chapter V, Lemma 8]{cassels1971}), we obtain the following
bound on the product of the heights of a basis for the subgroup $\aA \subseteq \uU_S(k)$.

\begin{corollary}\label{corintro2}  Let the multiplicative group of $S$-units $O_S^{\times}$ have positive rank $r$, and let
$\aA \subseteq \uU_S(k)$ be a subgroup of rank $r$.  Then there exists a basis $\gamma_1, \gamma_2, \dots , \gamma_r$ for 
the free group $\aA$, such that
\begin{equation}\label{intro93}
\prod_{j = 1}^r \bigl([k : \Q] h(\gamma_j)\bigr) \le \frac{2 (r!)^4}{(2r)!} \Reg_S(k) \bigl[\uU_S(k) : \aA \bigr]. 
\end{equation}
\end{corollary}

%%%%%%%%%%%%%%%%%%%%%%%%%%%%%%%%%%%%%%%%%%%%%%%%%%%%%%%%%%%%%%%
\section{Preliminary results}

At each place $v$ of $k$ we write $k_v$ for the completion of $k$ at $v$, so that $k_v$ is a local field.  We
select two absolute values $\|\ \|_v$ and $|\ |_v$ from the place $v$.  The absolute value $\|\ \|_v$ extends the usual archimedean or 
non-archimedean absolute value on the subfield $\Q$.  Then $|\ |_v$ must be a power of $\|\ \|_v$, and we set
\begin{equation}\label{int1}
|\ |_v = \|\ \|_v^{d_v/d},
\end{equation}
where $d_v = [k_v : \Q_v]$ is the local degree of the extension, and $d = [k : \Q]$ is the global degree.  With these normalizations
the height of an algebraic number $\alpha \not= 0$ that belongs to $k$ is given by
\begin{equation}\label{int3}
h(\alpha) = \sum_v \log^+ |\alpha|_v = \hh \sum_v \bigl|\log |\alpha|_v\bigr|,
\end{equation}
where where $\log^+ x= \max(0,\log x)$  for $x>0$.
Each sum in (\ref{int3}) is over the set of all places $v$ of $k$, and the equality between the two sums follows 
from the product formula.  Then $h(\alpha)$ depends on the algebraic number $\alpha \not= 0$, but it does not depend on 
the number field $k$ that contains $\alpha$.  We have already noted that the height is well defined as a map
\begin{equation*}\label{int4}
h : \G_k \rightarrow [0, \infty).
\end{equation*}
Elementary properties of the height show that the map $(\alpha, \beta) \mapsto h\bigl(\alpha \beta^{-1}\bigr)$ defines
a metric on the group $\G_k$.

Let $\eta_1, \eta_2, \dots , \eta_r$ be multiplicatively independent elements in $\uU_S(k)$ that form a basis for $\uU_S(k)$ as
a free abelian group of rank $r$.  Then let
\begin{equation*}\label{int15}
M = \bigl(d_v \log \|\eta_j\|_v\bigr)
\end{equation*}
denote the $(r+1)\times r$ real matrix, where $v \in S$ indexes rows and $j = 1, 2, \dots , r$ indexes columns.  At
each place $\widehat{v}$ in $S$ we write 
\begin{equation}\label{int17}
M^{(\widehat{v})} =  \bigl(d_v \log \|\eta_j\|_v\bigr)
\end{equation}
for the $r\times r$ submatrix of $M$ obtained by removing the row indexed by the place $\widehat{v}$.
Then the $S$-regulator of $O_S^{\times}$ (or of $\uU_S(k)$) is the positive number
\begin{equation}\label{int21}
\Reg_S(k) = \bigl|\det M^{(\widehat{v})}\bigr|,
\end{equation}
which is independent of the choice of $\widehat{v}$ in $S$.  Using an inequality
proved by A.~Schinzel \cite{schinzel1978} that bounds the determinant of a real matrix, we will prove that
\begin{equation}\label{int33}
\Reg_S(k) \le \prod_{j=1}^r \bigl([k : \Q] h(\eta_j)\bigr).
\end{equation}
If the better known inequality of Hadamard is used to estimate the determinant that defines the $S$-regulator on the 
right of (\ref{int21}), we obtain an upper bound that is larger than (\ref{int33}) by a factor of $2^r$.  

Assume more generally that $\alpha_1, \alpha_2, \dots , \alpha_r$ are multiplicatively independent elements in $\uU_S(k)$, but they do
not necessarily form a basis for the free group $\uU_S(k)$.  It follows that there exists an $r \times r$, nonsingular matrix 
$B = \bigl(b_{ij}\bigr)$ with entires in $\Z$, such that
\begin{equation}\label{int37}
\log \|\alpha_j\|_v = \sum_{i = 1}^r b_{ij} \log \|\eta_i\|_v
\end{equation} 
for each place $v$ in $S$ and for each $j = 1, 2, \dots , r$.  Alternatively, (\ref{int37}) can be written as the matrix identity
\begin{equation}\label{int41}
\bigl(d_v \log \|\alpha_j\|_v\bigr) = \bigl(d_v \log \|\eta_j\|_v\bigr) B.
\end{equation}
If
\begin{equation}\label{int47}
\aA = \langle \alpha_1, \alpha_2, \dots , \alpha_r\rangle \subseteq \uU_S(k)
\end{equation}
is the multiplicative subgroup generated by $\alpha_1, \alpha_2, \dots , \alpha_r$, we find that the index of this group is given by
\begin{equation}\label{int51}
\bigl[\uU_S(k) : \aA \bigr] = |\det B|.
\end{equation}
This will lead to the more general inequality (\ref{intro58}). 

%%%%%%%%%%%%%%%%%%%%%%%%%%%%%%%%%%%%%%%%%%%%%%%%%%%%%%%%%%%%%%%%%%%%%
\section{Relative regulators}

Throughout this section we suppose that  $k$ and $l$ are algebraic number fields with $k \subseteq l$.  We write $r(k)$ for the rank of the unit
group $O_k^{\times}$, and $r(l)$ for the rank of the unit group $O_l^{\times}$.  Then $k$ has $r(k) + 1$ 
archimedean places, and $l$ has $r(l) + 1$ archimedean places.  In general we have $r(k) \le r(l)$, and we recall 
(see \cite[Proposition 3.20]{narkiewicz2010}) that $r(k) = r(l)$ if and only if $l$ is a $\CM$-field, and $k$ is the maximal 
totally real subfield of $l$.  

The norm is a homomorphism of multiplicative groups
\begin{equation*}\label{unit1}
\Norm_{l/k} : l^{\times} \rightarrow k^{\times}.
\end{equation*}
If $v$ is a place of $k$, then each element $\alpha$ in $l^{\times}$ satisfies the identity
\begin{equation}\label{unit3}
[l : k] \sum_{w|v} \log |\alpha|_w = \log |\Norm_{l/k}(\alpha)|_v,
\end{equation}
where the absolute values $|\ |_v$ and $|\ |_w$ are normalized as in (\ref{int1}).
It follows from (\ref{unit3}) that the norm, restricted to the subgroup $O_l^{\times}$ of units, is a homomorphism
\begin{equation*}\label{unit5}
\Norm_{l/k} : O_l^{\times} \rightarrow O_k^{\times},
\end{equation*}
and the norm, restricted to the torsion subgroup in $O_l^{\times}$, is also a homomorphism
\begin{equation*}\label{unit7}
\Norm_{l/k} : \Tor\bigl(O_l^{\times}\bigr) \rightarrow \Tor\bigl(O_k^{\times}\bigr).
\end{equation*}
Therefore we get a well defined homomorphism, which we write as
\begin{equation*}\label{unit9}
\norm_{l/k} : O_l^{\times}/\Tor\bigl(O_l^{\times}\bigr) \rightarrow O_k^{\times}/\Tor\bigl(O_k^{\times}\bigr),
\end{equation*}
and define by
\begin{equation*}\label{unit11}
\norm_{l/k}\bigl(\alpha \Tor\bigl(O_l^{\times}\bigr)\bigr) = \Norm_{l/k}(\alpha) \Tor\bigl(O_k^{\times}\bigr).
\end{equation*}
However, to simplify notation we write
\begin{equation*}\label{unit15}
F_k = O_k^{\times}/\Tor\bigl(O_k^{\times}\bigr),\quad\text{and}\quad F_l = O_l^{\times}/\Tor\bigl(O_l^{\times}\bigr),
\end{equation*}
and we write the elements of the quotient groups $F_k$ and $F_l$ as coset representatives rather than cosets.
Obviously $F_k$ and $F_l$ are free abelian groups of rank $r(k)$ and $r(l)$, respectively. 

Following Costa and Friedman \cite{costa1991}, the subgroup of relative units in $O_l^{\times}$ is defined by
\begin{equation*}\label{unit19}
\big\{\alpha \in O_l^{\times} : \Norm_{l/k}(\alpha) \in \Tor\bigl(O_k^{\times}\bigr)\big\}.
\end{equation*}
Alternatively, we work in the free group $F_l$, where the image of the subgroup of relative units is the kernel of the homomorphism 
$\norm_{l/k}$.  That is, we define the subgroup of {\it relative units} in $F_l$ to be the subgroup
\begin{equation}\label{unit21}
E_{l/k} = \big\{\alpha \in F_l : \norm_{l/k}(\alpha) = 1\big\}.
\end{equation}
We also write
\begin{equation*}\label{unit23}
I_{l/k} = \big\{\textrm{norm}_{l/k}(\alpha) : \alpha \in F_l\big\} \subseteq F_k
\end{equation*}
for the image of the homomorphism $\norm_{l/k}$.  If $\beta$ in $F_l$ represents a coset in the subgroup $F_k$, then 
we have
\begin{equation*}\label{unit25}
\norm_{l/k}(\beta) = \beta^{[l : k]}.
\end{equation*}  
Therefore the image $I_{l/k} \subseteq F_k$ is a subgroup of rank $r(k)$, and the index satisfies
\begin{equation}\label{unit27}
[F_k : I_{l/k}] < \infty.
\end{equation}
It follows that $E_{l/k} \subseteq F_l$ is a subgroup of rank $r(l/k) = r(l) - r(k)$, and we restrict our attention here to extensions 
$l/k$ such that $r(l/k)$ is positive.    

Let $\eta_1, \eta_2, \dots , \eta_{r(l/k)}$ be a collection of multiplicatively independent relative units that form a basis for the subgroup 
$E_{l/k}$.  At each archimedean place $v$ of $k$ we select a place $\widehat{w}_v$ of $l$ such that $\widehat{w}_v | v$.  Then we 
define an $r(l/k) \times r(l/k)$ real matrix 
\begin{equation}\label{unit32}
M_{l/k} = \bigl([l_w : \Q_w] \log \|\eta_j\|_w\bigr),
\end{equation}
where $w$  is an archimedean place of $l$, but $w \not= \widehat{w}_v$ for each $v|\infty$, $w$ indexes rows, and 
$j = 1, 2, \dots , r(l/k)$ indexes columns.  We write $l_w$ for the completion of $l$ at the place $w$, $\Q_w$ for the
completion of $\Q$ at the place $w$, and we write $[l_w : \Q_w]$ for the local degree.  Of course $\Q_w$ is isomorphic
to $\R$ in the situation considered here.  As in \cite{costa1991}, we define the {\it relative regulator} of the extension $l/k$ to be 
the positive number
\begin{equation}\label{unit34}
\Reg\bigl(E_{l/k}\bigr) = \bigl|\det M_{l/k}\bigr|.
\end{equation}
It follows, as in the proof of \cite[Theorem 1]{costa1991} (see also \cite{costa1993}), that the
value of the determinant on the right of (\ref{unit34}) does not depend on the choice of places $\widehat{w}_v$ for 
each archimedean place $v$ of $k$.

\begin{theorem}\label{thmunit3}  Let $k \subseteq l$ be algebraic number fields such that the group $E_{l/k}$ of relative units
has positive rank $r(l/k) = r(l) - r(k)$.  Let $\vep_1, \vep_2, \dots , \vep_{r(l/k)}$ be a collection of multiplicatively independent relative units 
in $E_{l/k}$.  If $\ee \subseteq E_{l/k}$ is the multiplicative subgroup generated by $\vep_1, \vep_2, \dots , \vep_{r(l/k)}$, then
\begin{equation}\label{unit38}
\Reg\bigl(E_{l/k}\bigr) \bigl[E_{l/k} : \ee\bigr] \le \prod_{j = 1}^{r(l/k)} \bigl([l : \Q] h(\vep_j)\bigr).
\end{equation}
\end{theorem}

The relative regulator can also be expressed as a ratio of the (ordinary) regulators $\Reg(k)$ and $\Reg(l)$ by using the basic identity
\begin{equation}\label{unit45}
[F_k : I_{l/k}] \Reg(k) \Reg\bigl(E_{l/k}\bigr) = \Reg(l),
\end{equation}
which was established in \cite[Theorem 1]{costa1991}.  A slightly different definition for a relative regulator was considered
by Berg\'e and Martinet in \cite{berge1987}, \cite{berge1988}, and \cite{berge1989}.  We have used the definition proposed
by Costa and Friedman in \cite{costa1991} and \cite{costa1993}, as it leads more naturally to the inequality (\ref{unit38}).  Further
lower bounds for the product on the right of (\ref{unit38}) follow from inequalities for the relative regulator 
obtained by Friedman and Skoruppa \cite{friedman1999}.

%%%%%%%%%%%%%%%%%%%%%%%%%%%%%%%%%%%%%%%%%%%%%%%%%%%%%%%%%%%%%%%
\section{Schinzel's norm}

For a real number $x$ we write
\begin{equation*}\label{norm0}
x^+ = \max\{0, x\},\quad\text{and}\quad x^- = \max\{0, -x\},
\end{equation*}
so that $x = x^+ - x^-$ and $|x| = x^+ + x^-$.  If $\bx = (x_n)$ is a (column) vector in $\R^N$ we define
$$\delta : \R^N \rightarrow [0, \infty)$$ by
\begin{equation}\label{norm2}
\delta(\bx) = \max\bigg\{\sum_{m=1}^N x_m^+, ~\sum_{n=1}^N x_n^-\bigg\}.
\end{equation}
The following inequality was proved by A.~Schinzel \cite{schinzel1978}.

\begin{theorem}\label{thmnorm1}  If $\bx_1, \bx_2, \dots , \bx_N$, are column vectors in $\R^N$, then
\begin{equation}\label{norm3}
\bigl|\det\bigl(\bx_1\ \bx_2\ \cdots\ \bx_N\bigr)\bigr| \le \delta\bigl(\bx_1\bigr) \delta\bigl(\bx_2\bigr) \cdots \delta\bigl(\bx_N\bigr).
\end{equation}
\end{theorem}

\noindent An upper bound that is slightly sharper than (\ref{norm3}) was established by C.~R.~Johnson and M.~Newman \cite{johnson1980}.  
However, the bound obtained by Johnson and Newman does not lead to a significant improvement in the results we obtain here.

If $a$ and $b$ are nonnegative real numbers then
\begin{equation*}\label{norm4}
2 \max\{a, b\} = |a + b| + |a - b|.
\end{equation*}
This leads to the identity
\begin{equation}\label{norm5}
\delta(\bx) = \max\bigg\{\sum_{m=1}^N x_m^+, ~\sum_{n=1}^N x_n^-\bigg\}
	          = \hh \biggl|~\sum_{n=1}^N x_n\biggr| + \hh \sum_{n=1}^N |x_n|.
\end{equation}
It follows easily from (\ref{norm5}) that $\bx \mapsto \delta(\bx)$ is a continuous, symmetric distance function, or 
norm, defined on $\R^N$.  Let
\begin{equation}\label{norm9}
K_N = \big\{\bx \in \R^N : \delta(\bx) \le 1\big\}
\end{equation}
be the unit ball associated to the norm $\delta$.  Then $K_N$ is a compact, convex, symmetric subset of $\R^N$ having a nonempty interior.

\begin{lemma}\label{lemmink1}  Let $\delta : \R^N \rightarrow [0, \infty)$ be the continuous distance function defined by
{\rm (\ref{norm5})}, and let $K_N$ be the unit ball defined by {\rm (\ref{norm9})}.  Then we have
\begin{equation}\label{norm17}
\Vol_N(K_N) = \frac{(2N)!}{(N!)^3}.
\end{equation}
\end{lemma}

\begin{proof}  We write $J$ for the $(N+1)\times N$ matrix 
\begin{equation*}\label{norm21}
J = \hh \begin{pmatrix} 1 & 0 & 0 & \cdots & 0 \cr
                                       0 & 1 & 0 & \cdots & 0 \cr
                                       0 & 0 & 1 & \cdots & 0 \cr
                   \vdots & \vdots & \vdots &      & \vdots \cr
                                       0 & 0 & 0 & \cdots & 1 \cr
                                      -1 & -1 & -1 & \cdots & -1\cr\end{pmatrix}.                             
\end{equation*}  
Then it is obvious that $J$ has rank $N$.  Let
\begin{equation*}\label{norm25}
\D_N = \big\{\bwy \in \R^{N+1} : y_0 + y_1+ y_2 + \cdots + y_N = 0\big\},
\end{equation*}
so that $\D_N$ is the $N$-dimensional subspace of $\R^{N+1}$ spanned by the columns of $J$.  Further, let
\begin{equation*}\label{norm29}
\B_{N+1} = \big\{\bwy \in \R^{N+1} : \|\bwy\|_1 = |y_0| + |y_1| + |y_2| + \cdots + |y_N| \le 1\big\}
\end{equation*}
denote the unit ball in $\R^{N+1}$ with respect to the $\|\ \|_1$-norm.  If $\bx$ is a (column) vector in $\R^N$, we find that
\begin{equation*}\label{norm31}
\delta(\bx) = \|J\bx\|_1,
\end{equation*}
and therefore
\begin{equation*}\label{norm35}
K_N = \big\{\bx \in \R^N : \|J\bx\|_1 \le 1\big\}.
\end{equation*}
It follows that
\begin{align}\label{norm39}
\begin{split}
\Vol_N(K_N) &= \int_{\R^N} \chi_{\B_{N+1}}(J\bx)\ \dbx\\
		&= |\det U| \int_{\R^N} \chi_{\B_{N+1}}(JU\bx)\ \dbx,
\end{split}
\end{align}
where $\bwy \mapsto \chi_{\B_{N+1}}(\bwy)$ is the characteristic function of the subset $\B_{N+1}$, and $U$ is an arbitrary
$N\times N$ nonsingular real matrix.  We select $U$ so that the columns of the matrix $JU$ form an orthonormal basis for
the subspace $\D_N$.  With this choice of $U$ we find that
\begin{equation}\label{norm43}
\int_{\R^N} \chi_{\B_{N+1}}(JU\bx)\ \dbx = \Vol_N\bigl(\D_N \cap \B_{N+1}\bigr) = \frac{\sqrt{N+1} (2N)!}{2^N (N!)^3},
\end{equation}
where the second equality on the right of (\ref{norm43}) follows from a result of Meyer and Pajor \cite[Proposition II.7]{meyer1988}.
Because the columns of $JU$ are orthonormal, we get
\begin{equation}\label{norm53}
\bone_N = (JU)^T(JU).
\end{equation}
For each $m = 1, 2, \dots , N+1$ let $J^{(m)}$ be the $N\times N$ submatrix of $J$ obtained by removing the $m$-th row.
From (\ref{norm53}) and the Cauchy-Binet formula we have
\begin{align*}\label{norm57}
\begin{split}
1 &= \det\bigl((JU)^T(JU)\bigr)\\
   &= (\det U)^2 \det J^TJ \\ 
   &= (\det U)^2 \sum_{m=1}^{N+1} (\det J^{(m)})^2\\
   &= (\det U)^2 4^{-N} (N+1),                                                                                                                                                                                                                                                                                                                                                                                                                                                                                                                                                                                                                                                                                                                                                                                                                                                                       \end{split}
\end{align*}
and therefore
\begin{equation}\label{norm59}
|\det U| = \frac{2^N}{\sqrt{N+1}}. 
\end{equation}
The identity (\ref{norm17}) for the volume of $K_N$ follows by combining (\ref{norm39}), (\ref{norm43}) and (\ref{norm59}).
\end{proof}

Next we suppose that
\begin{equation*}\label{norm63}
A = \bigl(\ba_1\ \ba_2\ \cdots\ \ba_N\bigr)
\end{equation*}
is an $N\times N$ nonsingular matrix with columns $\ba_1, \ba_2, \dots , \ba_N$.  Obviously the columns of $A$ form
a basis for the lattice
\begin{equation}\label{norm67}
\eL = \big\{A\bxi : \bxi \in \Z^N\big\} \subseteq \R^N.
\end{equation}
Then by Schinzel's inequality we have
\begin{equation*}\label{norm71}
|\det A| \le \prod_{n=1}^N \delta(\ba_n).
\end{equation*}
Using the geometry of numbers, we will establish the existence of linearly independent points  $\bell_1, \bell_2, \dots , \bell_N$ in 
the lattice $\eL$, for which the product
\begin{equation*}\label{norm91}
\prod_{n=1}^N \delta(\bell_n)
\end{equation*}
is not much larger than $|\det A|$.  An explicit bound on such a product follows immediately from Minkowski's theorem on
successive minima and our formula (\ref{norm17}) for the volume of $K_N$.

\begin{theorem}\label{thmmink1}  Let $\eL \subseteq \R^N$ be the lattice defined by {\rm (\ref{norm67})}.  Then there 
exist linearly independent points $\bell_1, \bell_2, \dots , \bell_N$ in $\eL$ such that
\begin{equation}\label{norm93}
\prod_{n=1}^N \delta(\bell_n) \le \frac{2^N (N!)^3}{(2N)!} |\det A|.
\end{equation}
\end{theorem}

\begin{proof}  Let
\begin{equation*}\label{norm95}
0 < \lambda_1 \le \lambda_2 \le \cdots \le \lambda_N < \infty
\end{equation*}
be the successive minima of the lattice $\eL$ with respect to the convex symmetric set $K_N$.  Then there exist linearly
independent points $\bell_1, \bell_2, \dots , \bell_N$ in $\eL$ such that
\begin{equation*}\label{norm97}
\delta(\bell_n) = \lambda_n\quad\text{for each $n = 1, 2, \dots , N$.}
\end{equation*}
By Minkowski's theorem on successive minima (see \cite[section VIII.4.3]{cassels1971}) we have the inequality
\begin{equation*}\label{norm99}
\Vol_N(K_N) \lambda_1 \lambda_2 \cdots \lambda_N \le 2^N |\det A|.
\end{equation*}
From Lemma \ref{lemmink1} we get the bound
\begin{equation*}\label{norm101}
\prod_{n=1}^N \delta(\bell_n) \le \frac{2^N (N!)^3}{(2N)!} |\det A|,
\end{equation*}
and this proves the theorem.
\end{proof}

%%%%%%%%%%%%%%%%%%%%%%%%%%%%%%%%%%%%%%%%%%%%%%%%%%%%%%%%%%%%%%%
\section{Proof of Theorem \ref{thmintro1} and Theorem \ref{thmintro2}}

We require the following lemma, which connects the Schinzel norm (\ref{norm2}) with the Weil height.

\begin{lemma}\label{lemheight1}  Let $\widehat{v}$ be a place of the algebraic number field $k$, and let $\alpha \not= 0$ be an 
element of $k^{\times}$.  Then we have
\begin{equation}\label{height1}
\max\Big\{\sum_{v \not= \widehat{v}} \log^+ |\alpha|_v, \sum_{v \not= \widehat{v}} \log^- |\alpha|_v\Big\} = h(\alpha).
\end{equation}
\end{lemma}

\begin{proof}  The product formula implies that
\begin{equation*}\label{height3}
h(\alpha) = \sum_v \log^+ |\alpha|_v = \sum_v \log^- |\alpha|_v.
\end{equation*}
If $\log |\alpha|_{\widehat{v}} \le 0$ then
\begin{equation*}\label{height4}
\max\Big\{\sum_{v \not= \widehat{v}} \log^+ |\alpha|_v, \sum_{v \not= \widehat{v}} \log^- |\alpha|_v\Big\} = \sum_v \log^+ |\alpha|_v = h(\alpha).
\end{equation*}
On the other hand, if $\log |\alpha|_{\widehat{v}} \ge 0$ then
\begin{equation*}\label{height5}
\max\Big\{\sum_{v \not= \widehat{v}} \log^+ |\alpha|_v, \sum_{v \not= \widehat{v}} \log^- |\alpha|_v\Big\} = \sum_v \log^- |\alpha|_v = h(\alpha).
\end{equation*}
This proves the lemma.
\end{proof}

We now prove Theorem \ref{thmintro1}.  First we combine (\ref{int17}), (\ref{int21}), (\ref{int41}), and (\ref{int51}), and obtain the identity
\begin{equation}\label{height13}
 \Reg_S(k) \bigl[\uU_S(k) : \aA \bigr] = [k : \Q]^r  \bigl| \det(\log |\alpha_j|_v)\bigr|,
\end{equation}  
where $v$ in $S\setminus \{\widehat{v}\}$ indexes rows, and $j = 1, 2, \dots , r$ indexes columns, in the matrix on the right of (\ref{height13}).
We estimate the determinant in (\ref{height13}) by applying Schinzel's inequality (\ref{norm3}).  Using 
(\ref{norm2}) and (\ref{height1}) we get
\begin{align}\label{height17}
\begin{split}
\bigl| \det(\log |\alpha_j|_v)\bigr| &\le \prod_{j = 1}^r \max\Big\{\sum_{v \not= \widehat{v}} \log^+ |\alpha_j|_v, 
								\sum_{v \not= \widehat{v}} \log^- |\alpha_j|_v\Big\}\\
                                                          &= \prod_{j = 1}^r h(\alpha_j).
\end{split}
\end{align}
The inequality (\ref{intro58}) in the statement of Theorem \ref{thmintro1} follows from (\ref{height13}) and (\ref{height17}).

Next we prove Theorem \ref{thmintro2}.  Let $\eta_1, \eta_2, \dots , \eta_r$ be multiplicatively independent elements in $\uU_S(k)$ that 
form a basis for $\uU_S(k)$ as a free abelian group of rank $r$.  Let $\widehat{v}$ be a place of $k$ contained in $S$, and 
\begin{equation*}\label{height19}
M^{(\widehat{v})} =  \bigl(d_v \log \|\eta_j\|_v\bigr)
\end{equation*}
the $r\times r$ real matrix as defined in (\ref{int17}).  By hypothesis $\aA \subseteq \uU_S(k)$ is a subgroup of rank $r$.  Let 
$\alpha_1, \alpha_2, \dots , \alpha_r$ be multiplicatively independent elements in $\aA$ that form a basis 
for $\aA$.  As in (\ref{int37}), there exists an $r \times r$ nonsingular matrix $B = \bigl(b_{ij}\bigr)$ with entries in $\Z$, such that
\begin{equation}\label{height23}
\log \|\alpha_j\|_v = \sum_{i = 1}^r b_{ij} \log \|\eta_i\|_v
\end{equation} 
for each place $v$ in $S$ and for each $j = 1, 2, \dots , r$.  Alternatively, if we define the $r \times r$ real matrix
\begin{equation*}\label{height25}
A^{(\widehat{v})} = \bigl(d_v \log \|\alpha_j\|_v\bigr),
\end{equation*}
where $v$ in $S \setminus \{\widehat{v}\}$ indexes rows and $j = 1, 2, \dots , r$ indexes columns, then (\ref{height23}) is 
equivalent to the matrix identity
\begin{equation}\label{height27}
A^{(\widehat{v})} = M^{(\widehat{v})} B.
\end{equation}
We use the nonsingular $r \times r$ real matrix $A^{(\widehat{v})}$ to define a lattice $\eL^{(\widehat{v})} \subseteq \R^r$ by
\begin{equation*}\label{height31}
\eL^{(\widehat{v})} = \big\{A^{(\widehat{v})} \bxi : \bxi \in \Z^r\big\}.
\end{equation*}
Then (\ref{int21}), (\ref{int51}) and (\ref{height27}), imply that
\begin{equation}\label{height33}
\Reg_S(k) [\uU_S(k) : \aA] = \bigl|\det M^{(\widehat{v})}\bigr| |\det B| = \bigl|\det A^{(\widehat{v})}\bigr|,
\end{equation}
which is independent of the choice of $\widehat{v}$ in $S$, and is also the determinant of the lattice $\eL^{(\widehat{v})}$.  By 
Theorem \ref{thmmink1} and (\ref{height33}), there exist linearly independent vectors $\bell_1, \bell_2, \dots , \bell_r$ in 
$\eL^{(\widehat{v})}$ such that
\begin{equation}\label{height35}
\prod_{j=1}^r \delta(\bell_j) \le \frac{2^r (r!)^3}{(2r)!}  \Reg_S(k) [\uU_S(k) : \aA].
\end{equation}
As each (column) vector $\bell_j$ belongs to the lattice $\eL^{(\widehat{v})}$, it has rows indexed by the places $v$ in 
$S \setminus \{\widehat{v}\}$.  Thus $\bell_j$ can be written as
\begin{equation*}\label{height39}
\bell_j = \Bigl(d_v \sum_{i = 1}^r f_{ij} \log \|\alpha_i\|_v\Bigr) = \bigl(d_v \log \|\beta_j\|_v\bigr),
\end{equation*}
where $F = \bigl(f_{ij}\bigr)$ is an $r \times r$ nonsingular matrix with entires in $\Z$, and $\beta_1, \beta_2, \dots , \beta_r$ are
multiplicatively independent elements in the group $\aA$.  By Lemma \ref{lemheight1} we have
\begin{align}\label{height41}
\begin{split}
\delta(\bell_j) &=  \max\Big\{\sum_{v \not= \widehat{v}} d_v \log^+ \|\beta_j\|_v, \sum_{v \not= \widehat{v}} d_v \log^- \|\beta_j\|_v\Big\}\\
                         &= [k : \Q] \max\Big\{\sum_{v \not= \widehat{v}} \log^+ |\beta_j|_v, \sum_{v \not= \widehat{v}} \log^- |\beta_j|_v\Big\}\\
                         &= [k : \Q] h(\beta_j).
\end{split}
\end{align}
The inequality (\ref{intro88}) in the statement of Theorem \ref{thmintro2} follows from (\ref{height35}) and (\ref{height41}).

%%%%%%%%%%%%%%%%%%%%%%%%%%%%%%%%%%%%%%%%%%%%%%%%%%%%%%%%%%%%%%%
\section{Proof of Theorem \ref{thmunit3}}

Let $\eta_1, \eta_2, \dots , \eta_{r(l/k)}$ be a basis for the free abelian group $E_{l/k}$.  Then there exists a nonsingular, 
$r(l/k) \times r(l/k)$ matrix $C = \bigl(c_{ij}\bigr)$ with entries in $\Z$, such that
\begin{equation}\label{pf1}
\log \|\vep_j\|_w = \sum_{i = 1}^{r(l/k)} c_{ij} \log \|\eta_i\|_w
\end{equation}
at each archimedean place $w$ of $l$.  As in our derivation of (\ref{int41}) and (\ref{int51}), the equations (\ref{pf1}) can 
be written as the matrix equation
\begin{equation}\label{pf3}
\bigl([l_w : \Q_w] \log \|\vep_j\|_w \bigr) = \bigl([l_w : \Q_w] \log \|\eta_j\|_w\bigr) C,
\end{equation}
where $w$ is an archimedean place of $l$, and $w$ indexes the rows of the matrices on both sides of (\ref{pf3}).
Let $\ee$ be the subgroup of $E_{l/k}$ generated by $\vep_1, \vep_2, \dots , \vep_{r(l/k)}$.  It follows from (\ref{pf3}) 
that the index of $\ee$ in $E_{l/k}$ is given by
\begin{equation}\label{pf5}
\bigl[E_{l/k} : \ee\bigr] = |\det C|.
\end{equation}

At each archimedean place $v$ of $k$ let $\widehat{w}_v$ be a place of $l$ such that $\widehat{w}_v | v$.  As in (\ref{unit32}), we write
\begin{equation*}\label{pf9}
M_{l/k} = \bigl([l_w : \Q_w] \log \|\eta_j\|_w\bigr),
\end{equation*}
for the $r(l/k) \times r(l/k)$ matrix, where $w$  is an archimedean place of $l$, but $w \not= \widehat{w}_v$ 
for each $v|\infty$, $w$ indexes rows, and $j = 1, 2, \dots , r(l/k)$ indexes columns.  Let
\begin{equation*}\label{pf13}
L(\ee) = \bigl([l_w : \Q_w] \log \|\vep_j\|_w\bigr)
\end{equation*}
be the analogous $r(l/k) \times r(l/k)$ matrix, where again $w$  is an archimedean place of $l$, but $w \not= \widehat{w}_v$ 
for each $v|\infty$, $w$ indexes rows, and $j = 1, 2, \dots , r(l/k)$ indexes columns.  From (\ref{pf3}) we get the matrix identity
\begin{equation}\label{pf17}
L(\ee) = M_{l/k} C.
\end{equation}
Then we combine (\ref{unit34}), (\ref{pf3}), (\ref{pf5}), and (\ref{pf17}), and conclude that
\begin{equation}\label{pf19}
\Reg\bigl(E_{l/k}\bigr) \bigl[E_{l/k} : \ee\bigr] = \bigl|\det L(\ee)\bigr|.
\end{equation}

To complete the proof we apply Schinzel's inequality (\ref{norm3}) to the determinant on the right of (\ref{pf19}).
We find that
\begin{align}\label{pf23}
\begin{split}
[l : \Q]^{-r(l/k)} \bigl|\det L(\ee)\bigr| &\le \prod_{j=1}^{r(l/k)} \bigg\{\hh\Bigl|\sum_{w \not= \widehat{w}_v} \log |\vep_j|_w\Bigr| 
			+ \hh \sum_{w \not= \widehat{w}_v} \bigl|\log |\vep_j|_w\bigr|\bigg\}\\
			                                   &= \prod_{j=1}^{r(l/k)} \bigg\{\hh\Bigl|\sum_{v|\infty} \log |\vep_j|_{\widehat{w}_v}\Bigr| 
			+ \hh \sum_{w \not= \widehat{w}_v} \bigl|\log |\vep_j|_w\bigr|\bigg\}\\
			                                   &\le \prod_{j=1}^{r(l/k)} \bigg\{\hh \sum_{v|\infty} \bigl|\log |\vep_j|_{\widehat{w}_v}\bigr| 
			+ \hh \sum_{w \not= \widehat{w}_v} \bigl|\log |\vep_j|_w\bigr|\bigg\}\\
			                                   &= \prod_{j = 1}^{r(l/k)} h(\vep_j).
\end{split}
\end{align}
Combining (\ref{pf19}) and the inequality (\ref{pf23}), leads to the bound (\ref{unit38}) in the statement of Theorem \ref{thmunit3}.

%%%%%%%%%%%%%%%%%%%%%%%%%%%%%%%%%%%%%%%%%%%%%%%%%%%%%%%%%%%%%%%

%\today
\end{document}